\documentclass{amsart}
\usepackage{amsmath,amsthm,amsfonts,listings,bbm,amssymb,bm,dsfont}
\usepackage{graphicx,tikz,color}
\usepackage[colorlinks=true, allcolors=blue]{hyperref}
\numberwithin{equation}{section}

\newcommand{\m}{\mathfrak{m}_{\lambda,\alpha}}
\newcommand{\s}{\sigma}
\newcommand{\sk}{\mathfrak{s}_{\lambda,\alpha}}
\newcommand{\al}{\alpha}
\newcommand{\la}{\lambda}
\newcommand{\te}{\theta}
\newcommand{\K}[2]{K_{#1,#2}}
\newcommand{\n}[2]{\phi_{#1,#2}}
\newcommand{\nExpr}[2]{\frac{\Gamma(1+#1)}{\Gamma(1+#1+\alpha)}\frac{\Gamma(#2 +\alpha + #1)}{\Gamma(#2+#1)}}
\newcommand{\F}{\mathcal{F}}

\newtheorem{theorem}{Theorem}

\newtheorem*{theorem*}{Theorem}
\newtheorem{proposition}{Proposition}
\newtheorem{lemma}{Lemma}

\title{A Martingale Approach to Large--$\theta$ Ewens--Pitman Model}
\author{Rodrigo Ribeiro$^1$ }
\date{March 2025}

\begin{document}

\maketitle

\footnote{University of Denver - Colorado, USA. \textsc{rodrigo.ribeiro@du.edu}}
\begin{abstract}
    We investigate the asymptotic behavior of the number of parts $K_n$ in the Ewens--Pitman partition model under the regime where the diversity parameter is scaled linearly with the sample size, that is, $\theta = \lambda n$ for some~$\lambda > 0$. While recent work has established a law of large numbers (LLN) and a central limit theorem (CLT) for $K_n$ in this regime, we revisit these results through a martingale-based approach. Our method yields significantly shorter proofs, and leads to sharper convergence rates in the CLT, including improved Berry--Esseen bounds in the case $\alpha = 0$, and a new result for the regime $\al \in (0,1)$, filling a gap in the literature.
\end{abstract}

\section{Introduction}

The \emph{Ewens--Pitman model} is a fundamental probability distribution on set partitions, introduced by Pitman~\cite{Pitman1995} as a two-parameter generalization of the classical one-parameter Ewens sampling formula from population genetics~\cite{Ewens1972}. In this model, given two parameters $0\le \al < 1$ and $\te > -\al$, a random partition $\Pi_n$ of the set $\{1,\dots,n\}$  having $K_n$ parts of sizes $N_n = (N_1,\dots,N_{K_n})$ is drawn from the following distribution
\begin{equation}\label{eq:modelpmf}
    P\left( K_n = k, N_n = (n_1,\dots,n_k \right) = \frac{1}{k!}\binom{n}{n_1,\dots,n_k}\frac{[\te]_{k,\al}}{[\te]_{n,1}}\prod_{i=1}^n[1-\al]_{n_i-1},
\end{equation}
where $[x]_{m,a}$ is the rising factorial of $x$ of order $m$ and increment $a$.
 When $\alpha=0$, the model reduces to the classical Ewens model, closely related to Kingman's Poisson--Dirichlet distribution~\cite{Kingman1975}.

The model has become key in several fields. In \emph{population genetics} \cite{Volkov2003}, it models the allele frequency spectrum in a neutral population. In \emph{Bayesian nonparametrics} \cite{Johnson2007,Teh2006}, it underpins priors such as the Pitman--Yor process used for flexible mixture models. Applications also extend to \emph{machine learning} \cite{Blei2011}, \emph{combinatorics} (partition structures), and \emph{statistical physics}. See~\cite{Pitman2006} for an extensive overview.

A central quantity of interest is the number of parts $K_n$ in the random partition. In the standard regime with fixed parameters, the asymptotic behavior of $K_n$ as $n \to \infty$ is well understood. For $\alpha=0$, $K_n \sim \theta \log n$ almost surely, with Gaussian fluctuations of order $\sqrt{\log n}$~\cite{Korwar1973}. For $0 < \alpha < 1$, $K_n$ scales as $n^{\alpha}$, and the normalized variable converges to a non-degenerate \emph{Mittag--Leffler} limit~\cite{Pitman2006}. These results include strong laws of large numbers (SLLNs), central limit theorems (CLTs), and Berry--Esseen-type refinements~\cite{Dolera2020,Bercu2024}.

In practice, it is natural to consider scenarios where $\theta$ grows with $n$. For instance, in population genetics, $\theta$ corresponds to a scaled mutation rate and may increase with sample size. The regime $\theta = \lambda n$, for fixed $\lambda > 0$, is particularly interesting. This  scaling was first studied by Feng~\cite{Feng2007} for $\alpha = 0$, who derived large deviation principles for $K_n$ in the Ewens model. Recently, Contardi, Dolera, and Favaro~\cite{Contardi2024} extended this analysis to the full Ewens--Pitman model with $0 \le \alpha < 1$, proving LLNs and CLTs for $K_n$ under the linear-$\theta$ regime. 


\subsection{Main Results}

Before we state our main results, let us introduce some notation. For $m  \in \mathbb{N}$, let $\K{\al,\te}{m}$ be the number of parts in the Ewens-Pitman model over $\{1,2,\dots, m\}$ with parameters $\al$ and $\te$. Also, let $\m$ and $\sk^2$ be the following quantities
\begin{equation}\label{def:m}
    \m := \begin{cases} 
    \frac{\lambda}{\alpha}\left[\left( 1 + \frac{1}{\lambda}\right)^{\alpha} -1 \right], &\; \text{if } \alpha \in (0,1) \\
    \lambda\log\left( 1 +\frac{1}{\lambda} \right),& \; {\text{if }\alpha = 0}
    \end{cases}
\end{equation}
and
\begin{equation}\label{def:sk}
    \sk^2 := \begin{cases}
    \frac{\la}{\al}\left[ \left( 1 + \frac{1}{\la}\right)^{2\al}\left( 1 -\frac{\al}{1+\la} \right) -\left(1+\frac{1}{\la}\right)^\al\right], & \text{for }\al \in (0,1) \\
     \lambda\log\left( 1 +\frac{1}{\lambda} \right) - \frac{\lambda}{1+\lambda},& \text{for }\al = 0.
\end{cases}
\end{equation}
Moreover, let $\Phi$ denote the CDF of a standard normal distribution. We show the following results for the linear-$\te$ regime:
\begin{theorem}[Strong Law of Large Numbers]\label{t:lln} Fix $\la > 0$, then for $\al \in [0,1)$
$$
\lim_{n \to \infty}\frac{\K{\al,\la n}{n}}{n} = \m,
$$
almost surely.
\end{theorem}
We also show a Central Limit Theorem with Berry-Esseen-type bounds

\begin{theorem}[CLT and Berry-Esseen-type bounds]\label{t:clt}Fix $\la > 0$ and $\al \in [0,1)$, and let$F_n$ be the CDF of
$$
\frac{\sqrt{n}}{\sk}\left( \frac{\K{\al,\la n}{n}}{n} - \m\right),
$$
Then, for any $\epsilon \in (0,1)$ there exists constant $C_1 = C(\la,\al, \epsilon)$ and $C_2 = C(\la)$, such that
$$
\sup_{x\in \mathbb{R}}|F_n(x) - \Phi(x)| \le \begin{cases}
    C_1n^{-1/5 + 2\epsilon/5}, & \text{for }\al \in (0,1) \\
    C_2n^{-1/2}, & \text{for }\al = 0.
\end{cases} 
$$
for all $n$. In particular,
$$
\frac{\sqrt{n}}{\sk}\left( \frac{\K{\al,\la n}{n}}{n} - \m\right) \xrightarrow{d} \mathcal{N}(0,1).
$$
    
\end{theorem}

\subsection{Our Contribution}
Although the LLN and CLT results for $K_{\al,\te,n}$ under $\theta = \lambda n$ have firstly appeared in the preprint~\cite{Contardi2024}, our paper present several improvements:
\begin{itemize}
  \item \textbf{Substantially shorter proofs}: While the original analysis in \cite{Contardi2024} spans a total of $54$ pages, our martingale method leads to concise arguments totaling $19$ pages. This simplifies technical arguments, making the results more accessible.
  \item \textbf{Sharper convergence rates}: We establish improved Berry--Esseen bounds for the Central Limit Theorem, enhancing the convergence rate obtained in \cite{Contardi2024} from $n^{1/8}$ to $n^{1/2}$, in the case $\al=0$, which in general is the best rate possible for CLT.
  \item {\bf New Berry-Esseen-type bounds for $\al \in (0,1)$} Our refinement for the CTL also extends to the case $\al \in (0,1)$, with bounds of order $n^{-1/5 + \epsilon}$. This fills a gap in the literature.
\end{itemize}


\subsection{Discussion on the Rate of Convergence}

Notice that in Theorem~\ref{t:clt}, we have a difference in the convergence rates of the Central Limit Theorem depending on the value of \(\alpha\).

When \(\alpha=0\), the random variable \(K_{0,\la n,n}\) can be expressed as a sum of independent Bernoulli random variables, (see Section~\ref{s:gcrp}). Consequently, the classical Berry-Esseen theorem applies, yielding an error bound of order \(1/\sqrt{n}\), as expected.

In contrast, for \(\alpha \in (0,1)\) the process exhibits a reinforcement mechanism.
More precisely, the increments in this regime depend on the current state of the process, (see Section~\ref{s:gcrp}). This dependency creates a martingale structure with non-independent increments, complicating the control of conditional variances, which are random. Even if the total variance is normalized, its fluctuations need to be controlled, which introduces additional error terms. These extra terms generally worsen the rate because they cannot be as tightly controlled as the constant variance in the independent case. This is why, in such settings, Berry-Esseen techniques typically yield slower convergence rates.

\subsection{Proof ideas: A General Principle.}

Our paper illustrates how a shift in perspective can simplify proofs and yield stronger results. Instead of directly manipulating the Ewens-Pitman model's probability mass function  defined at \eqref{eq:modelpmf}, or any other stastic representation of it—which can be computationally intensive \cite{Contardi2024}—we use a sequential construction known as the \emph{Generalized Chinese Restaurant} \cite{Crane16,oliveira2022concentration,Pitman2006}. From this viewpoint, the statistic $\K{\alpha,\theta}{n}$ emerges as a sequence of bounded increment random variables.

Central to our proofs is a general principle of handling bounded increment processes. For such a process $\{S_n\}_n$, Doob's Decomposition Theorem provides a decomposition
\[
S_n = M_n + A_n,
\]
where $M_n$ is a bounded increment martingale and $A_n$ is predictable. Leveraging martingale concentration inequalities (e.g., Azuma's and Freedman's inequalities) on $M_n$ often simplifies analysis, as the predictable component $A_n$ typically has a straightforward structure.

Alternatively, one may transform a bounded increment process $\{S_n\}_n$ into a martingale via a suitable transformation $T$, preserving increment boundedness. This approach directly exploits martingale results such as concentration inequalities and convergence theorems, with the extra advantage that there is no predictable component. However, identifying such a transformation may be challenging.

In our work, we adopt the latter approach by constructing a suitable transformation $T$ of $K$ into a martingale. We apply Azuma's inequality and a Borel-Cantelli argument to establish a Law of Large Numbers, showing $T(K_n)/n$ converges to a constant. For Berry-Esseen-type bounds, we utilize results for mean-zero martingales based on conditional variances and increments, focusing primarily on carefully controlling these quantities. The control of the conditional variance is where most of our work is.

\subsection{Paper Organization}
At Section~\ref{s:gcrp}, we introduce the {\it Generalized Chinese Restaurant Proccess} which provides a sequential construction of the Ewens-Pitman model. We also show the existence of a transformation of $\K{\al,\la n}{n}$ that gives us a martingale with all the properties we need. We then, at Section \ref{s:lln}, show the law of large numbers for $\K{\al,\la n}{n}$ given by Theorem~\ref{t:lln}, by leveraging the martingale obtained at the previous section. At Section~\ref{s:clt}, we show the CLT with Berry-Esseen-type bounds. 

\section{The Generalized Chinese Restaurant and a useful martingale}\label{s:gcrp}
The motor behind our proofs is a shift in perspective. Rather than working with the original defition of the Ewens-Pitman model, we opt for seeing it as the so-called {\it Generalized Chinese Restaurant} \cite{aldous85}, which we define below.

Let $0 \le \alpha < 1$ and $\theta > -\alpha$. Consider a restaurant with an infinite number of tables. Customers arrive sequentially and choose tables according to the following rule:
\begin{itemize}
    \item The first customer always sits at the first table.
    \item Suppose there are currently $k$ occupied tables with $n_i$ customers seated at table $i$, where $1 \le i \le k$ and $\sum_{i=1}^{k} n_i = n-1$. The $n$-th customer then either:
    \begin{enumerate}
        \item Joins an existing table $i$ with probability
        \[
        \frac{n_i - \alpha}{\theta + n - 1}, \quad 1 \leq i \leq k,
        \]
        \item or opens a new $(k+1)$-th table with probability
        \[
        \frac{\theta + \al k}{\theta + n - 1}.
        \]
    \end{enumerate}
\end{itemize}
An important property of the GCRP is that its sequential procedure generates exchangeable random partitions whose distribution are exactly the Ewens--Pitman partition model with parameters $(\alpha,\theta)$, see \cite{oliveira2022concentration, Pitman2006}. That is, if $\mathcal{P}_n$ denotes a random partition of $\{1,2,\dots,n\}$ sampled from the Ewens--Pitman model with parameters $(\al,\te)$, then we can also sample $\mathcal{P}_n$ by running $n$ steps of the GCRP-$(\al,\te)$ and letting each occupied table represent a set in our partition. In this analogy, the number of parts in $\mathcal{P}_n$ is exactly the number of occupied tables.

From the definition of the GCRP we can observe a crucial distinction between the regimes $\al = 0$ and $\al \in (0,1)$. In the regime $\al = 0$, $\K{\al,\te}{n}$ is a sum of independent Bernoulli random variables, since at the $j$-th step, the $j$-th customer opens a new table with probability $\te/(\te + j-1)$ independently of the past.

On the other hand, when $\al \in (0,1)$, $\{\K{\al,\te}{j}\}_j$ is a reinforced process, since conditionally to $\F_{\al,\te,j}$, the expected increment $\K{\al,\te}{j+1} - \K{\al,\te}{j}$ is proportional to $\K{\al,\te}{j}$. This difference will force us to treat the case $\al = 0$ separately in some situations.

We move now towards stating and proving the main results of this section. Some additional notation will be needed. We being by dropping the $\al$ in our notation, since it will remain fixed. Thus, from now on, we will be writing just $\K{\te}{n}$ for the number of parts. For $\al \in [0,1)$, we will also let $\n{\te}{j}$ be the following quantity
\begin{equation}\label{def:phi}
    \n{\theta}{j} := \prod_{i=1}^{j-1}\left( 1+ \frac{\alpha}{i+\theta} \right) = \nExpr{\theta}{j}, \quad \text{for }j\ge2
\end{equation}
and $\phi_{\te,1} = 1$.  When $\al = 0$, $\n{\te}{j} = 1$ for all $j$ and $\te$. We also define $\psi_{\te,j}$ and $Z_{\te,j}$ as
\begin{equation}\label{def:psiZ}
    \psi_{\te,j} := (\te + \al)\n{\te}{j} \quad \text{and}\quad Z_{\te,j} := \te + \al \K{\te}{j}.
\end{equation}
Finally, given any process $\{X_j\}_j$, we let $\Delta X_j$ be
\begin{equation}\label{def:DeltaX}
    \Delta X_{j} := X_{j} - X_{j-1}.
\end{equation}
We are now ready for the main result of this section.
\begin{proposition}[A useful martingale]\label{p:useful_martingale} The process $\{Z_{\theta,j}/\psi_{\te,j}\}_{j\in \mathbb{N}}$ is a martingale with respect to the natural filtration $\{\mathcal{F}_{\te,j}\}_{j \in \mathbb{N}}$. Additionally, the following  hold 
\begin{enumerate}
    \item $EZ_{\te,j} = \psi_{\te,j}$, for all $j \in \mathbb{N}$. 
    \item For all $j$ the following bound holds almost surely 
        $$
            \left|\frac{Z_{\te,j+1}}{\psi_{\te,j+1}} - \frac{Z_{\te,j}}{\psi_{\te,j}}\right| \le \frac{2}{\psi_{\te,j+1}};
        $$

    \item For all $j$
    $$
        E\left[ \left( \frac{Z_{\te,j+1}}{\psi_{\te,j+1}} - \frac{Z_{\te,j}}{\psi_{\te,j}} \right)^2 \; \middle | \; \F_{\te,j} \right]  = \frac{\al^2}{\psi^2_{\te,j+1}}\frac{Z_{\te,j}}{(\te+j)}\left( 1 - \frac{Z_{\te,j}}{\te+j} \right),
    $$
    almost surely.
\end{enumerate}
    
\end{proposition}
\begin{proof} 
    To see that $Z_{\te,j}$ properly normalized forms a martingale, notice that by the definition of $Z_{\te,j}$ and $\K{\te}{j}$, it follows that
    \begin{equation*}
        \begin{split}
            E\left[ \Delta Z_{\te,j+1}  \; \middle | \; \mathcal{F}_{\te,j} \right] & = \al E\left[ \Delta \K{\te}{j+1}\; \middle | \; \mathcal{F}_{\te,j} \right] = \frac{\al(\te + \al \K{\te}{j})}{\te + j} = \frac{\al Z_{\te,j}}{\te + j}.
        \end{split}
    \end{equation*}
The above identity then implies that
$$
E\left[  Z_{\te,j+1}  \; \middle | \; \mathcal{F}_{\te,j} \right] = \left( 1 + \frac{\al}{\te + j}\right)Z_{\te,j}.
$$
Recalling the definition of $\psi_{\te,j}$, we can divide both sides of the above identity by $\psi_{\te,j+1}$ to obtain
$$
E\left[  \frac{Z_{\te,j+1}}{\psi_{\te,j+1}} \; \middle | \; \mathcal{F}_{\te,j} \right] =  \frac{Z_{\te,j}}{\psi_{\te,j}},
$$
which shows that $\{Z_{\theta,j}/\psi_{\te,j}\}_{j\in \mathbb{N}}$ is a martingale. To see that $\psi_{\te,j}$ is the expected value of $Z_{\te,j}$, observe that $Z_{\te,1} = \te + \al$ (because $\K{\te}{1} \equiv 1$), and that $EZ_{\te,1}/\psi_{\te,1} = 1$.

The upper bound on $\left|\frac{Z_{\te,j+1}}{\psi_{\te,j+1}} - \frac{Z_{\te,j}}{\psi_{\te,j}}\right|$ is a consequence of the definitions of $Z_{\te,j}$ and $\psi_{\te,j}$ together with the fact that $\Delta \K{\te}{j+1}$ is either $0$ or $1$ and that $\alpha \in [0,1].$

As for item $3$, notice that
\begin{equation*}
    \begin{split}
        E\left[ \left( \frac{Z_{\te,j+1}}{\psi_{\te,j+1}} - \frac{Z_{\te,j}}{\psi_{\te,j}} \right)^2 \; \middle | \; \F_{\te,j} \right] &= E\left[ \left( \frac{\Delta Z_{\te,j+1} - \frac{\al Z_{\te,j}}{\te+j}}{\psi_{\te,j+1}}  \right)^2 \; \middle | \; \F_{\te,j} \right] \\
        & = \frac{1}{\psi_{\te,j+1}^2}{\rm Var}\left[ \Delta Z_{\te,j+1} \; \middle | \; \F_{\te,j} \right] \\
       & = \frac{\al^2}{\psi_{\te,j+1}^2}{\rm Var}\left[ \Delta \K{\te}{j+1} \; \middle | \; \F_{\te,j} \right] \\
       & = \frac{\al^2}{\psi^2_{\te,j+1}}\frac{Z_{\te,j}}{(\te+j)}\left( 1 - \frac{Z_{\te,j}}{\te+j} \right),
    \end{split}
\end{equation*}
since conditioned on $\F_{\te,j}$, $\Delta\K{\te}{j+1}$ obeys a Bernoulli distribution with parameter $ Z_{\te,j}/(\te + j)$. This concludes the proof of the lemma.
\end{proof}

\section{Strong Law of Large Numbers for $\K{\la n}{n}$}\label{s:lln}
In this section, we prove Theorem \ref{t:lln}. The reader might ask themselves, why this is not a simple application of martingale convergence theorems. The answer is that, as a process on $n$, $\{Z_{\la n,n}/\psi_{\la n,n}\}_{n}$ does not constitute a martingale. Thus, we will need an intermediate step in our proof. Moreover, in order to get an explicit formula for the limit, we will need a technical lemma. It will be helpful to recall the definition of $\m$ given at \eqref{def:m} on Page \pageref{def:m}.

\begin{lemma}\label{l:sum_phi_asymp}For all $\alpha \in [0,1)$, it holds that
\begin{enumerate}
    \item $$
    \left| \n{\la n}{n} - \left( 1 + \frac{1}{\lambda} \right)^{\alpha}\right| = O\left( \frac{1}{n} \right);
    $$
    
    \item 
    $$ 
    \left | \frac{\n{\la n}{n}}{n}\sum_{i=1}^n\frac{\lambda n}{(\lambda n +i)\n{\la n}{i}} - \m \right| =  O\left( \frac{1}{n} \right).
    $$ 
\end{enumerate}
    
\end{lemma}
We postpone the proof of the above result to the end of this section. For now, we focus on showing Theorem~\ref{t:lln}.
\begin{proof}[Proof of Theorem \ref{t:lln}] We begin adverting the reader that throughout this proof and the others to come, we will let $C$ denote a constant depending only on $\al$ and $\la$ (and possibly some auxiliary $\epsilon$) that may change from line to line.
\newline

\noindent{\underline{Case $\al \in (0,1)$}.}
Now, setting $\theta = \lambda n$, $j=n$, by Proposition \ref{p:useful_martingale} we have that
$$
\frac{Z_{\la n,j}}{\psi_{\la n, j}} = \frac{Z_{\la n,j}}{(\la n +\al)\phi_{\la n, j}}
$$
as a process on $j$ is a martingale whose increments are bounded by
\begin{equation}\label{eq:boundDeltaZpsi0}
    \begin{split}
    \left|\frac{Z_{\la n,j+1}}{\psi_{\la n,j+1}} - \frac{Z_{\la n,j}}{\psi_{\la n,j}}\right| & \le \frac{2}{\psi_{\la n,j+1}} \\
    & = \frac{2\Gamma(\la n +1 + \al)\Gamma(\la n + j +1 )}{(\la n +\al)\Gamma(\la n +1)\Gamma(\la n + j + 1+ \al )}. 
    \end{split}
\end{equation}
Using the following expansion showed by F.G. Tricomi and A. Erd\'{e}lyi in \cite{tricomi1951asymptotic}:
\begin{equation}\label{eq:gamma_expansion}
    \frac{\Gamma(z+a)}{\Gamma(z+b)} = z^{a-b}\left( 1+ \frac{(a-b)(a+b-1)}{2z} + O\left(|z|^{-2}\right) \right).
\end{equation}
taking $z$ as $\la n +j$ and $\la n$ on \eqref{eq:boundDeltaZpsi0} yields
\begin{equation}\label{eq:boundDeltaZpsi}
    \begin{split}
    \left|\frac{Z_{\la n,j+1}}{\psi_{\la n,j+1}} - \frac{Z_{\la n,j}}{\psi_{\la n,j}}\right| & \le \frac{(\la n)^\al}{(\la n +\al)(\la n + j)^\al}\left( 1+ O(n^{-1}) \right)\\
    & \le  \frac{C}{(\la n + \al)^{1-\al}(\la n + j)^{\al}},
    \end{split}
\end{equation}
for some positive constant $C$ that depends on $\la$ and $\al$ only. Bounding the sum by $n$ times its largest term, we can obtain that
\begin{equation}\label{eq:boundDeltaZsquared}
    \sum_{j=1}^{n-1} \left|\frac{Z_{\la n,j+1}}{\psi_{\la n,j+1}} - \frac{Z_{\la n,j}}{\psi_{\la n,j}}\right|^2 \le \frac{C}{n},
\end{equation}
for some constant depending on $\la$ and $\al$ only. Thus, by Azuma's inequality
$$
P\left( \left|\frac{Z_{\la n,n}}{\psi_{\la n,n}} - 1\right| \ge \varepsilon \right) \le 2\exp \left \lbrace - \frac{\varepsilon^2n}{2C} \right \rbrace,
$$
which, by the first Borel-Cantelli lemma, implies that $\frac{Z_{\la n,n}}{\psi_{\la n,n}}$ converges to $1$ a.s. as $n$ goes to infinity. To show the convergence of $\K{\la n}{n}$ we use the convergence we have just proved, the definition of $Z_{\la n,n}$ and the first part of the technical Lemma~\ref{l:sum_phi_asymp} to obtain
$$
\frac{\K{\la n}{n}}{(\la n + \al)\phi_{\la n,n}} \longrightarrow \frac{1}{\al} - \frac{\la^\al}{\al(1+\la)^\al}.
$$
With the above convergence in mind, we can write
$$
\frac{\K{\la n}{n}}{n} = \frac{(\la n + \al)\phi_{\la n, n}}{n}\frac{\K{\la n}{n}}{(\la n + \al)\phi_{\la n, n}} \longrightarrow \la\left(1+\frac{1}{\la} \right)^\al \left[ \frac{1}{\al} - \frac{\la^\al}{\al(1+\la)^\al} \right],
$$
which finishes the proof, since the limit is exactly $\m$, as desired.
\newline

\noindent{\underline{Case $\al = 0.$}} Recall the discussion at Section \ref{s:gcrp} on Page \pageref{s:gcrp}. In this regime, $\K{\la n}{n}$ is the sum of $n$ independent Bernoulli random variables
\begin{equation}\label{def:Kalzero}
    \K{\la n}{n} = \sum_{j=1}^{n-1} {\rm Ber}\left( \frac{\la n}{\la n+j}\right).
\end{equation}
Then, by the Strong Law of Large Numbers for sum of independent random variables with uniformly bounded fourth moment, we have that
$$
\frac{\K{\la n}{n}-E\K{\la n}{n}}{n} \to 0,
$$
almost surely. The proof then follows from Lemma \ref{l:sum_phi_asymp}, Part $2$, together with the fact that when $\al = 0$, $\n{\la n}{i} = 1$ for all $i$ and $n$, which implies that
$$
 \frac{1}{n}\sum_{i=1}^n\frac{\lambda n}{\lambda n +i} = \frac{E\K{\la n}{n}}{n} \to \frak{m}_{\la,0},
$$
at a rate of $O(1/n)$.
\end{proof}
\subsection{Proof of Technical Lemma \ref{l:sum_phi_asymp}}
To conclude the proof of Theorem \ref{t:lln}, we show Lemma \ref{l:sum_phi_asymp}, but before, we will show a simple and known result that will be very useful in other parts of the paper.
\begin{lemma}[Riemann Sum Error Bound]\label{l:integral_approx}
Let \( f: [0,1] \to \mathbb{R} \) be a positive $C^1$ function having a bounded derivative. Then, for any Riemann sum $M_n$ with $\Delta x=1/n$, the following holds
\[
\left| M_n - \int_0^1f(x){\rm d}x \right| \le \frac{\sup_{x\in [0,1]}|f'(x)|}{n}.
\]
\end{lemma}

\begin{proof}
Let \( \Delta x = 1/n \) and partition the interval \([0,1]\) into \( n \) subintervals
\[
\left[\frac{i-1}{n}, \frac{i}{n}\right], \quad i=1,2,\dots,n.
\]
By the Integral Mean Value Theorem, it follows that
$$
\int_{(i-1)/n}^{i/n}f(x){\rm d}x = \frac{f(\xi_i)}{n},
$$
for some $\xi_i$ in the $i$-th subinterval. Then, for any Riemann sum $M_n$ with $\Delta x = 1/n$, triangle inequality gives us
$$
\left| M_n - \int_0^1f(x){\rm d}x\right| = \left| \frac{1}{n}\sum_{i=1}^nf(x_i^*) - f(\xi_i) \right| \le \frac{1}{n}\sum_{i=1}^n|f(x_i^*) - f(\xi_i)|.
$$
Finally, by the Mean Value Theorem, we have that
$$
|f(x_i^*) - f(\xi_i)| \le \frac{|f'(x_i^{**})|}{n},
$$
which combined to the previous inequality proves the result.
\end{proof}
We can finally prove Lemma~\ref{l:sum_phi_asymp}.
\begin{proof}[Proof of Lemma \ref{l:sum_phi_asymp}] We start from part $1$. \\ 

\noindent{\underline{Proof of Part $(1)$}} This part is a consequence of the definition of $\n{\la n}{n}$ and the expansion of the ration of gamma functions given at \eqref{eq:gamma_expansion}. Indeed,
$$
\n{\la n}{n} = \nExpr{\lambda n}{n} = \frac{1}{\la^\al n^\al}(1+\la)^\al n^{\al}\left( 1+ O\left( \frac{1}{n}\right) \right),
$$
which shows this part of the lemma for $\al \in (0,1)$. As for $\al =0$, the lemma is trivially satisfied, since in this case, $\n{\la n}{n} = 1$. \\

\noindent{\underline{Proof of Part $(2)$}} In this part, we will rely again on the expansion given at \eqref{eq:gamma_expansion}. Using the expression of $\n{n}{j}$ in terms of $\Gamma$, we obtain that
\begin{equation}\label{eq:sum_with_gamma}
            \frac{\n{\la n}{n}}{n}\sum_{i=1}^n\frac{\lambda n}{(\lambda n +i)\n{\la n}{i}}   = \frac{\la\n{\la n}{n}\Gamma(\la n+\al + 1)}{\Gamma(\la n+ 1)}\sum_{i=1}^n\frac{\Gamma(\la n + i)}{(\la n + i)\Gamma(\la n +i+\al) }.
\end{equation}
Now, we will focus on the asymptotics of the sum on the RHS of the above equality. By \eqref{eq:gamma_expansion} with $z= \la n + i$, $b=\al$ and $a = 0$, yields
\begin{equation}\label{eq:sum_expanded}
    \sum_{i=1}^n\frac{\Gamma(\la n + i )}{(\la n + i)\Gamma(\la n +i+\al) } = \sum_{i=1}^n\frac{1}{(\la n + i)^{1+\al}}\left( 1+ \frac{\al(1-\al)}{2(\la n + i)} + O\left((\la n + i)^{-2}\right) \right).
\end{equation}
Notice that, for $\al \in [0,1)$,
\begin{equation*}
    \begin{split}
        \sum_{i=1}^n\frac{1}{(\la n + i)^{1+\al}} = \frac{1}{n^\al}\sum_{i=1}^n\frac{1}{n(\la + i/n)^{1+\al}} =: \frac{1}{n^\al}R_n,
    \end{split}
\end{equation*}
where $R_n$ is a right Riemann sum of the function $f(x) = 1/(\la + x)^{1+\al}$ over $[0,1]$. Now, let $\frak{a}_{\la,\al}$ be the following quantity
\begin{equation}\label{def:a}
    \frak{a}_{\la,\al} := \begin{cases}
        \frac{1}{\al}\left( \frac{1}{\la^\al} - \frac{1}{(1+\la)^\al}\right), & \text{when }\al \in (0,1)\\
        \log(1+\la) - \log(\la), & \text{when }\al=0.
    \end{cases}
\end{equation}
Then, by Lemma~\ref{l:integral_approx}
\begin{equation*}
    R_n = \int_0^1\frac{{\rm d}x}{(\la + x)^{1+\al}} + O(n^{-1}) = \frak{a}_{\la,\al} + O(n^{-1}).
\end{equation*}
We can then get back to \eqref{eq:sum_expanded} and write
$$
\sum_{i=1}^n\frac{\Gamma(\la n + i +1)}{(\la n + i)\Gamma(\la n +i+1+\al) } = \frac{\frak{a}_{\la,\al}}{n^{\al}} + O(n^{-1-\al}),
$$
where we bounded the other sums that appear in \eqref{eq:sum_expanded} by their integral counterparts to obtain terms of order at most $n^{-1-\al}$. Finally, getting back to \eqref{eq:sum_with_gamma}, and expanding the factors in $\Gamma$ again using \eqref{eq:gamma_expansion}, yields 
$$
\frac{\n{\la n}{n}}{n}\sum_{i=1}^n\frac{\lambda n}{(\lambda n +i)\n{\la n}{i}} = \la\cdot\frac{(1+\la)^\al}{\la^\al}\cdot \la^\al n^\al\cdot \frac{\frak{a}_{\la,\al}}{ n^\al} + O(n^{-1}), 
$$
which proves the lemma.
\end{proof}

\section{Central Limit Theorem for $\K{\la n}{n}$}\label{s:clt}
In this section, we provide a proof of Theorem~\ref{t:clt}. The key technical ingredient in our argument is Theorem~\ref{t:hallheyde}, which gives explicit bounds on the rate of convergence in the CLT for martingale sequences. Our approach mirrors that of the Strong Law of Large Numbers: we first establish the result for the normalized statistic $Z_{\lambda n, n}/\psi_{\lambda n, n}$, and subsequently extend the conclusion to $K_n^{(\lambda n)}$ through appropriate scaling and shifting transformations.

Throughout this section, we will have many different processes and indexes, thus, to avoid clutter, let us begin by introducing some additional notation. We introduce the constants
\begin{equation}\label{def:sigma}
    \s^2_{\la,\al} := \frac{\al}{\la}\left(1 - \frac{\la^\al}{(1+\la)^\al} - \frac{\al}{1+\la}\right) \; \text{and} \;\sk^2 := \begin{cases}
        \frac{\la^2\s^2_{\la,\al}(1+1/\la)^{2\al}}{\al^2},& \al \in(0,1);\\
        \la \log\left(\frac{1+\la}{\la}\right) - \frac{\la}{1+\la}, &\al = 0,
        \end{cases}
\end{equation}
together with following processes
\begin{equation}\label{def:YjXj}
  Y_{n,j}:= \frac{Z_{\la n,j}}{\psi_{\la n,j}} \quad \text{and}\quad  X_{n,j} := \frac{\sqrt{n}}{\s_{\la,\al}}\left(Y_{n,j} - Y_{n,j-1} \right).
\end{equation}
Recall that $EY_{n,j}=1$, for all $j$ and $n$, and that, for fixed $n$, $\{X_{n,j}\}_j$ is a martingale difference. Also let $S_{n,j}$ be
\begin{equation}\label{def:Sn}
    S_{n,j} = \frac{\sqrt{n}}{\s_{\la,\al}}\left(Y_{n,j} -1\right) = \sum_{i=1}^jX_{n,i},
\end{equation}
which is a mean-zero martingale, as a process on $j$, by virtue of  Proposition \ref{p:useful_martingale}. Finally, we introduce the conditional variance of $S_{n,j}$
\begin{equation}\label{def:Vn2}
    V_n^2 = \sum_{j=1}^nE\left[ X_{n,j}^2\; \middle | \; \F_{n,j}\right].
\end{equation}

It is important to highlight that only $\sk$ has a definition for $\al=0$. The other quantities and processes are defined for $\al \in (0,1)$ only.
The proof of Theorem \ref{t:clt} will follow from the two results below
\begin{proposition}[Berry-Esseen bounds for $S_{n,n}$]\label{p:beSn}Let $F_n$ be the CDF of $S_{n,n}$. Then, for any $\epsilon > 0 $ there exists a constant $C$ depending on $\epsilon, \la$ and $\al$ only, such that
$$
\sup_{-\infty < x < \infty}|F_n(x) - \Phi(x)| \le Cn^{-1/5 + \epsilon/5}.
$$
\end{proposition}
The second result is a technical lemma whose proof is quite short, for this reason, we state it and prove it right away.
\begin{lemma}\label{l:philip}Let $(a_n)_n,(b_n)_n$ be two sequences of real numbers such that $a_n \to 0$ and $b_n \to 1$. Then,
$$
\sup_{x\in \mathbb{R}}|\Phi((x-a_n)/b_n) -\Phi(x)| = O(|a_n| + |1-b_n|)
$$
\end{lemma}
\begin{proof} The result is a consequence of the Mean Value Theorem. In fact, for a given $x$, the Mean Value Theorem guarantees the existence of some $t_{x,n}\in [(x-a_n)/b_n,x]$ (or in $[x,(x-a_n)/b_n]$), such that
\begin{equation*}
    \begin{split}
        |\Phi((x-a_n)/b_n) -\Phi(x)| &  = \frac{1}{\sqrt{2\pi}}e^{-t_{x,n}^2/2}\frac{|x(1-b_n) -a_n|}{|b_n|}\\ 
        & \le C(|x|e^{-t^2_{x,n}/2}|1-b_n| + |a_n|),
    \end{split}
\end{equation*}
for some absolute constant $C$. This is enough to conclude the lemma, since for $|x|<M$, we can bound $|x|$ by $M$, and in the case $|x|\ge M$, the factor $|x|e^{-t_{x,n}^2}$ is also bounded, since $t_{x,n}\in [(x-a_n)/b_n,x]$ (or in $[x,(x-a_n)/b_n]$).
\end{proof}
As for the proof of Proposition \ref{p:beSn}, we postpone it to the end of this section. For now, we will focus on showing how our main result follows from it and Lemma \ref{l:philip}.
\begin{proof}[Proof of Theorem \ref{t:clt}] We begin by the case $\al \in (0,1)$.\\

\noindent{\underline{Case $\al\in(0,1)$.}}
The proof follows by a proper transformation of $\K{\la n}{n}$. Indeed, 
\begin{equation}\label{eq:first_algebra}
    \begin{split}
        \frac{\sqrt{n}}{\sk}\left( \frac{\K{\la n}{n}}{n} - \m\right) & = \frac{\sqrt{n}}{\sk}\left( \frac{\al\K{\la n}{n}}{\al n} - \m\right)\\
        & = \frac{\sqrt{n}}{\sk}\left( \frac{\la n + \al\K{\la n}{n}}{\al n} - \m - \frac{\la}{\al}\right)\\
        &= \frac{\sqrt{n}}{\sk}\left[ \frac{\psi_{\la n,n}}{\al n}\left(\frac{\la n + \al\K{\la n}{n}}{\psi_{\la n,n}} - 1\right) \right]\\
        & \quad + \frac{\sqrt{n}}{\sk}\left[ \frac{\psi_{\la n, n}}{\al n} -\m - \frac{\la}{\al}\right].
    \end{split}
\end{equation}
We will take care of the second term in the right hand side of the last equation. Recall that $\psi_{\la n,j} = (\la n + \al)\n{\la n}{n}$ and that for $\al \in (0,1)$
$$
\m =  \frac{\lambda}{\alpha}\left[\left( 1 + \frac{1}{\lambda}\right)^{\alpha} -1 \right],
$$
which combined imply
    \begin{equation*}
        \begin{split}
        \frac{\sqrt{n}}{\sk}\left( \frac{\psi_{\la n, n}}{\al n} -\m - \frac{\la}{\al}\right) & = \frac{\sqrt{n}}{\sk}\left[ \frac{(\la n + \al)\phi_{\la n, n}}{\al n} -\frac{\la}{\al}\left(1+\frac{1}{\la}\right)^\al\right]\\
       & = \frac{\sqrt{n}}{\sk}\left[\frac{\la}{\al}\left( \phi_{\la n,n} - \left(1+\frac{1}{\la}\right)^\al \right) + \frac{\phi_{\la n,n}}{n}\right] 
        \end{split}
    \end{equation*}
Thus, a triangle inequality and Part $1$ of the technical Lemma \ref{l:sum_phi_asymp}, yield
\begin{equation}\label{eq:Osqrtn}
    \left| \frac{\sqrt{n}}{\sk}\left( \frac{\psi_{\la n, n}}{\al n} -\m - \frac{\la}{\al}\right)  \right| = O\left( \frac{1}{\sqrt{n}} \right).
\end{equation}
As for the first term in the  right hand side of \eqref{eq:first_algebra}, we have
\begin{equation}\label{eq:s}
    \begin{split}
     \frac{\sqrt{n}}{\sk}\left[ \frac{\psi_{\la n,n}}{\al n}\left(\frac{\la n + \al\K{\la n}{n}}{\psi_{\la n,n}} - 1\right) \right] & = \frac{\s_{\la,\al}\psi_{\la n,n}}{\sk \al n}\left[ \frac{\sqrt{n}}{\s_{\la,\al}}\left(\frac{\la n + \al\K{\la n}{n}}{\psi_{\la n,n}} - 1\right) \right] \\
     & = \frac{\s_{\la,\al}\psi_{\la n,n}}{\sk \al n}S_{n,n}.
    \end{split}
\end{equation}
Combining the above with \eqref{eq:Osqrtn} and returning to \eqref{eq:first_algebra} we can write
\begin{equation}\label{eq:knO}
 \frac{\sqrt{n}}{\sk}\left( \frac{\K{\la n}{n}}{n} - \m\right) = \frac{\s_{\la,\al}\psi_{\la n,n}}{\sk \al n}S_{n,n} + O\left( \frac{1}{\sqrt{n}} \right),
\end{equation}
which by Proposition \ref{p:beSn}, technical Lemma \ref{l:sum_phi_asymp} and Slutsky's theorem give us that  the process $\frac{\sqrt{n}}{\sk}\left( \frac{\K{\la n}{n}}{n} - \m\right)$ converges in distribution to a standard normal, since
$$
\frac{\s_{\la,\al}\psi_{\la n,n}}{\sk\al n} = \frac{\s_{\la,\al}(\la n+\al)\phi_{\la n,n}}{\sk \al n} \to \frac{\frac{\la \s_{\la,\al}(1+1/\la)^\al}{\al}}{\sk} = 1.
$$
The rate of convergence requires some extra work. Let $F_{K_n}$ and $F_{S_n}$ denote the CDF of $\frac{\sqrt{n}}{\sk}\left( \frac{\K{\la n}{n}}{n} - \m\right)$ and $S_{n,n}$, respectively. Notice that by \eqref{eq:knO}, we have
\begin{equation}\label{eq:FkFs}
    F_{K_n}(x) = F_{S_n}\left(\frac{x-a_n}{b_n} \right),
\end{equation}
where 
\begin{equation}\label{eq:anbn}
a_n = O\left( \frac{1}{\sqrt{n}}\right) \quad \text{and} \quad b_n = \frac{\s_{\la,\al}\psi_{\la n,n}}{\sk\al n}.
\end{equation}
Now, for a given $x$, an application of triangle inequality together with Proposition~\ref{p:beSn} and technical Lemma \ref{l:philip} leads us to
\begin{equation*}
    \begin{split}
        |F_{K_n}(x) - \Phi(x)| & = \left|F_{S_n}\left(\frac{x-a_n}{b_n}\right) - \Phi(x)\right|\\
        & \le \left|F_{S_n}\left(\frac{x-a_n}{b_n}\right) - \Phi\left(\frac{x-a_n}{b_n}\right)\right| + \left|\Phi\left(\frac{x-a_n}{b_n}\right)-\Phi(x)\right| \\
        & \le An^{-1/5 + \epsilon/5} +O\left(\frac{1}{\sqrt{n}}\right) + O(|1-b_n|).\\
    \end{split}
\end{equation*}
To conclude the proof, just notice that by technical Lemma \ref{l:sum_phi_asymp} and the definition of~$\sk$ given at~\eqref{def:sk} on Page~\pageref{def:sk}, we have
\begin{equation*}
    \begin{split}
        \left|1 - \frac{\s_{\la,\al}\psi_{\la n,n}}{\sk\al n}\right| &= \left|1 - \frac{\s_{\la,\al}(\la n + \al) \phi_{\la n,n}}{\sk\al n}\right| \\
        & =  \left|1 - \frac{\s_{\la,\al}\left(\frac{\la}{\al}+ \frac{\al}{n}\right) \phi_{\la n,n}}{\sk}\right|\\
        & \le \frac{1}{\sk}\left|\s_{\la,\al}\frac{\la}{\al}\left[\left(1+\frac{1}{\la} \right)^\al - \phi_{\la n,n} \right] \right| + O\left( \frac{1}{n}\right)\\
        & = O\left( \frac{1}{n}\right),
    \end{split}
\end{equation*}
which allows us to obtain
$$
|F_{K_n}(x) - \Phi(x)| \le An^{-1/5+ \epsilon/5} + O\left(\frac{1}{\sqrt{n}}\right) + O\left(\frac{1}{n}\right) = O\left(n^{-1/5+ \epsilon/5}\right),
$$
as desired. \\

\noindent{\underline{Case $\al = 0$.}} In this regime, we are under essentially classical settings: $\K{\la n}{n}$ is sum of independent Bernoulli random variables. For this reason, we will not fill all the details. The proof is done by applying the classical result due to Petrov stated at Theorem~\ref{t:petrov}.

In order to apply Theorem~\ref{t:petrov}, let us introduce the following notation
$$
X_{n,j} := {\rm Ber}\left( \frac{\la n}{\la n + j}\right) - \frac{\la n}{\la n + j}; \quad \s_{n,j}^2 := EX_{n,j}^2; \quad \s_n^2 := \sum_{j=1}^n\s_{n,j}^2
$$
and 
$$
L_n := \s_n^{-3/2}\sum_{j=1}^nE|X_{n,j}|^3 \le n\s_{n}^{-3/2},
$$
almost surely, since $X_{n,j} \le 1$. Notice that approximating the sum by the integral using Lemma \ref{l:integral_approx}, gives us that
\begin{equation*}
    \begin{split}
        \frac{\s_n^2}{n} & = \frac{1}{n}\sum_{j=1}^n\frac{\la n}{\la n +j} - \frac{(\la n)^2}{(\la n + j)^2}\\
        & = \la\sum_{j=1}^n \frac{1}{n}\cdot \frac{1}{(\la + j/n)} - \la^2\sum_{j=1}^n\frac{1}{n}\cdot\frac{1}{(\la + j/n)^2}\\ 
        & = \la \int_0^1\frac{{\rm d}x}{\la +x} - \la^2\int_0^1\frac{{\rm d}x}{(\la +x)^2} +O\left(\frac{1}{n}\right)\\
        & = \la \log(1+1/\la) - \frac{\la}{1+\la} + O\left(\frac{1}{n}\right),
    \end{split}
\end{equation*}
which implies that
\begin{equation}\label{eq:sn2}
\s_n^2 = \frak{s}^2_{\la,0}n + O(1).
\end{equation}
Thus, by Theorem \ref{t:petrov}, we have that
\begin{equation*}
    \sup_{x\in \mathbb{R}}\left| P\left( \frac{\K{\la n}{n}-E\K{\la n}{n}}{\s_n} \le x\right)  - \Phi(x)\right| \le Cn^{-1/2},
\end{equation*}
for some constant $C = C(\la)$. Extending the result to $\sqrt{n}/\frak{s}_{\la,0}(\K{\la n}{n}/n - \m)$ is very similar to the argument given in the case $\al \in (0,1)$. It is enough to notice that by Lemma \ref{l:sum_phi_asymp}, Part $2$, and \eqref{eq:sn2} we have that
$$
|E\K{\la n}{n}/n - \frak{m}_{\la,0}| = O(n^{-1}) \quad \text{and} \quad \left| \frac{\s_n}{\frak{s}_{\la,0}\sqrt{n}}-1\right| = O\left( \frac{1}{\sqrt{n}} \right)
$$
which combined to Lemma \ref{l:philip} gives us a rate of convergence of order~$1/\sqrt{n}$ as desired.
\end{proof}

\subsection{Proof of Proposition \ref{p:beSn}}
We conclude this section by presenting the proof of Proposition~\ref{p:beSn}. The proposition follows from three intermediate results. For two of these results, all necessary ingredients are already at our disposal, and we will therefore present their proofs immediately. The proof of the third result, however, involves more technical arguments, and thus we defer it to the end of this section.

The first result is a concentration inequality for the variables $Y_{n,j}$.
\begin{lemma}\label{l:concentation_Ynj} Let $Y_{n,j}$ be as in \eqref{def:YjXj}. Then, there exists a positive constant $C$ depending on $\la$ and $\al$ only, such that
$$
P\left(  \left|Y_{n,j} - 1\right| > \varepsilon, \text{ for some }j\le n\right) \le 2ne^{-C\varepsilon^2n}
$$
\end{lemma}
\begin{proof} We already have all we need for the proof. Firstly, recall from Proposition~\ref{p:useful_martingale} that $\{Y_{n,j}\}_{j}$ is a bounded-increment martingale of mean $1$. Secondly, by \eqref{eq:boundDeltaZsquared}, for any $j\le n$
$$
\sum_{i=1}^j|\Delta Y_{n,i}|^2 \le \frac{C}{n},
$$
for some positive constant $C$ depending on $\la$ and $\al$ only. Finally, by Azuma's inequality
$$
P\left( \left|Y_{n,j} - 1\right| > \varepsilon\right) \le 2e^{-\varepsilon^2n/2C}.
$$
The result follows by union bound over $j\le n$.
\end{proof}
The second intermediate result we need gives us an expression for $V_n^2$ in terms of $Y_{n,j}$'s.
\begin{lemma}\label{l:Vnexpre}Let $V_n^2$ be as in \eqref{def:Vn2}. Then,
$$
V_n^2 = \frac{\al^2n}{\s_{\la,\al}^2}\sum_{j=1}^n \frac{Y_{n,j}}{(\la n+j+\al)\psi_{\la n,j+1}} - \frac{Y^2_{ n,j}}{(\la n+j+\al)^2}
$$
    
\end{lemma}
\begin{proof} Again, we already have all we need, and the proof will be just a matter of putting all the pieces together. By part $3$ of Proposition \ref{p:useful_martingale}, we have that 
\begin{equation}\label{eq:X2}
    \begin{split}
    E\left[X_{n,j+1}^2\; \middle | \F_{n,j}\right] & =\frac{n}{\s_{\la,\al}^2} E\left[ \left( \frac{Z_{\la n,j+1}}{\psi_{\la n,j+1}} - \frac{Z_{\la n,j}}{\psi_{\la n,j}} \right)^2 \; \middle | \; \F_{\la n,j} \right] \\
    & = \frac{\al^2n}{\s_{\la,\al}^2}\frac{Z_{\la n,j}}{(\la n+j)\psi^2_{\la n,j+1}}\left( 1 - \frac{Z_{\la n,j}}{\la n+j} \right)
    \end{split}
\end{equation}
Recall the definition of $\phi_{\te,j}$ given at \eqref{def:phi} and the one of $\psi_{\te,j}$ given at \eqref{def:psiZ}, it follows that
\begin{equation}\label{eq:phipsi}
    \psi_{\la n, j+1} = \left(1+\frac{\al}{\la n +j}\right)\psi_{\la n,j}.
\end{equation}
Thus, getting back to \eqref{eq:X2}, we can write
\begin{equation*}
    \begin{split}
    E\left[X_{n,j+1}^2\; \middle | \F_{n,j}\right]  & = \frac{\al^2n}{\s_{\la,\al}^2}\frac{Y_{n,j}}{(\la n+j+\al)\psi_{\la n,j+1}}\left( 1 - \frac{Z_{\la n,j}}{\la n+j} \right)\\
    &=\frac{\al^2n}{\s_{\la,\al}^2}\frac{Y_{n,j}}{(\la n+j+\al)\psi_{\la n,j+1}} - \frac{\al^2n}{\s_{\la,\al}^2}\frac{Y^2_{ n,j}}{(\la n+j+\al)^2},
    \end{split}
\end{equation*}
which concludes the proof.
\end{proof}
The third result, is a technical estimate, whose proof we defer to the end of this section.
\begin{lemma}\label{l:sn} Let $\s_{\la,\al}$ be as in \eqref{def:sigma}, then 
    $$
    \left|\al^2n\sum_{j=1}^n\left[\frac{1}{(\la n+j+\al)\psi_{\la n,j+1}} - \frac{1}{(\la n+j+\al)^2}\right] - \s_{\la,\al}^2\right| = O\left(n^{-1}\right).
    $$
    Moreover, the sequences $n\sum_{j=1}^n\frac{1}{(\la n+j+\al)\psi_{\la n,j+1}}$ and $n\sum_{j=1}^n \frac{1}{(\la n+j+\al)^2}$ also converge separately.
\end{lemma}
We are now ready for the proof of Proposition \ref{p:beSn}. 
\begin{proof}[Proof of Proposition \ref{p:beSn}] The result will follow from Theorem~\ref{t:hallheyde}, which means that we have obtain an upper bound for the quantity
\begin{equation}\label{def:Ln}
L_n = \sum_{i=1}^nE|X_{n,j}|^{2+2\delta} + E|V_n^2-1|^{1+\delta},
\end{equation}
for some $\delta \in (0,1]$. We begin bounding the summation in the RHS of $L_n$. In order to do that, recall the bound at \eqref{eq:boundDeltaZpsi} which gives us that
\begin{equation*}
    \begin{split}
        |X_{n,j}|^{2+2\delta} &\le \frac{n^{1+\delta}}{\s_{\la,\al}^{2+2\delta}}\frac{C}{(\la n + \al)^{(1-\al)(2+2\delta)}(\la n +j)^{\al(2+2\delta)}} \\
        & \le \frac{C}{n^{(1+\delta)(1-2\al)}(\la n +j)^{\al(2+2\delta)}}
        \end{split}
\end{equation*}
Then, bounding the sum by the integral, yields
\begin{equation}\label{eq:boundXnj}
    \begin{split}
     \sum_{j=1}^n |X_{n,j}|^{2+2\delta} \le Cn^{-\delta},
    \end{split}
\end{equation}
for another constant depending only on $\la, \al$ and also on $\delta$. 
    
The next step is to bound $E|V_n^2-1|^{1+\delta}$. For that, we will again introduce some additional notation. Let $\s_n^2$ be the following
\begin{equation*}
    \s^2_n := \al^2n\sum_{j=1}^n\left[\frac{1}{(\la n+j+\al)\psi_{\la n,j+1}} - \frac{1}{(\la n+j+\al)^2}\right].
\end{equation*}
Then, setting $\delta = 1$ and applying the technical Lemma \ref{l:sn} we have the following bound
\begin{equation}\label{eq:boundEVn1}
    \begin{split}
        E|V_n^2 - 1|^2 & = E\left(V_n^2 -\frac{\s^2_n}{\s^2_{\la,\al}} + \frac{\s^2_n}{\s_{\la,\al}}  -1\right)^2\\ 
        & \le 2E\left(V_n^2 -\frac{\s^2_n}{\s^2_{\la,\al}}\right)^2 + \frac{2}{\s_{\la,\al}}\left(\s^2_n - \s^2_{\la,\al}\right)^2\\
        & \le 2E\left(V_n^2 -\frac{\s^2_n}{\s^2_{\la,\al}}\right)^2 + O(n^{-2}).
    \end{split}
\end{equation}
We now focus on deriving an upper bound for the first term in the right hand side of the last inequality. Notice that by Lemma \ref{l:Vnexpre} and the definition of $\s_n$, we have that
\begin{equation}\label{eq:boundVnminuss}
    \begin{split}
        \left| V_n^2 - \frac{\s^2_n}{\s^2_{\la,\al}} \right | & = \left|\frac{\al^2n}{\s_{\la,\al}^2}\sum_{j=1}^n \frac{Y_{n,j}-1}{(\la n+j+\al)\psi_{\la n,j+1}} - \frac{Y^2_{ n,j}-1}{(\la n+j+\al)^2}\right| \\
        & \le \frac{\al^2n}{\s_{\la,\al}^2}\sum_{j=1}^n \frac{|Y_{n,j}-1|}{(\la n+j+\al)\psi_{\la n,j+1}} + \frac{|Y^2_{ n,j}-1|}{(\la n+j+\al)^2}.
    \end{split}
\end{equation}
Now, fixed $\epsilon \in (0,1)$, let $A_n$ be the following event
$$
A_n = \left \lbrace \left|Y_{n,j} - 1\right| > \frac{1}{n^{1/2 -\epsilon}}, \text{ for some }j\le n\right \rbrace.
$$
Then, on the event $A_n^c$, we have that
\begin{equation}\label{eq:boundparti}
    \begin{split}
    \frac{\al^2n}{\s_{\la,\al}^2}\sum_{j=1}^n \frac{|Y_{n,j}-1|}{(\la n+j+\al)\psi_{\la n,j+1}} & \le \frac{1}{n^{1/2-\epsilon}}\frac{\al^2n}{\s_{\la,\al}^2}\sum_{j=1}^n \frac{1}{(\la n+j+\al)\psi_{\la n,j+1}} \\ 
    &\le \frac{C}{n^{1/2 -\epsilon}},
    \end{split}
\end{equation}
for some constant $C$ depending on $\la$ and $\al$ only, due to the fact that by technical Lemma \ref{l:sn}, $\frac{\al^2n}{\s_{\la,\al}^2}\sum_{j=1}^n \frac{1}{(\la n+j+\al)\psi_{\la n,j+1}}$ converges as $n$ goes to infinity.

As for the second summation of \eqref{eq:boundVnminuss}, we handle it in a similar manner. First. notice that $Y_{n,j}$ is bounded. In fact, by \eqref{eq:boundDeltaZpsi}, we have that
$$
|Y_{n,j}|  \le \sum_{i=1}^j|\Delta Y_{n,i}| \le  \sum_{i=1}^j\frac{C}{(\la n + \al)^{1-\al}(\la n + j)^{\al}} \le C,
$$
for some possibly different constant $C=C(\la,\al)$. Thus,
\begin{equation}\label{eq:boundpartii}
    \begin{split}
        \frac{\al^2n}{\s_{\la,\al}^2}\sum_{j=1}^n\frac{|Y^2_{ n,j}-1|}{(\la n+j+\al)^2} & = \frac{\al^2n}{\s_{\la,\al}^2}\sum_{j=1}^n\frac{|Y_{ n,j}-1||Y_{n,j}+1|}{(\la n+j+\al)^2} \\
        & \le \frac{C\al^2n}{\s_{\la,\al}^2}\sum_{j=1}^n\frac{|Y_{ n,j}-1|}{(\la n+j+\al)^2} \\ 
        &\le \frac{C}{n^{1/2-\epsilon}}\frac{\al^2n}{\s_{\la,\al}^2}\sum_{j=1}^n\frac{1}{(\la n+j+\al)^2} \\
       & \le \frac{C}{n^{1/2 -\epsilon}},
    \end{split}
\end{equation}
since $\frac{\al^2n}{\s_{\la,\al}^2}\sum_{j=1}^n\frac{1}{(\la n+j+\al)^2}$ converges. Plugging \eqref{eq:boundparti} and \eqref{eq:boundpartii} back on \eqref{eq:boundVnminuss}, we obtain the following on the event $A_n^c$
$$
 \left| V_n^2 - \frac{\s^2_n}{\s^2_{\la,\al}} \right | = O\left(^{n^{-1/2+\epsilon}} \right).
$$
Now, observing that $V_n^2$ is bounded ( it can be trivially bounded by a polynomial in $n$ if one does not want to be very precise), and apply Lemma \ref{l:concentation_Ynj} which gives that $P(A_n) \le 2n\exp\{-Cn^{-\epsilon}\}$, we can conclude
\begin{equation*}
    \begin{split}
        E\left| V_n^2 - \frac{\s^2_n}{\s^2_{\la,\al}} \right |^2 \le  E\left[\left| V_n^2 - \frac{\s^2_n}{\s^2_{\la,\al}} \right |^2\mathds{1}_{A_n^c}\right] + Mn\exp\{-Cn^{-\epsilon/2}\} = O\left(^{n^{-1+2\epsilon}} \right),
    \end{split}
\end{equation*}
for some constant $M$. Finally, combining the above with \eqref{eq:boundXnj} (setting $\delta=1$), and using Theorem \ref{t:hallheyde} yields
$$
|P(S_{n,n}\le x) - \Phi(x)| \le CL_n^{1/5}\left[1+|x|^{16/5}\right]^{-1} \le \frac{C}{n^{1/5 -2\epsilon/5}},
$$
as desired.
\end{proof}

\subsection{Proof of Technical Lemma \ref{l:sn}}
Finally, we are left with the proof of Lemma~\ref{l:sn} to conclude the proof of Proposition~\ref{p:beSn}.
\begin{proof}[Proof of Lemma \ref{l:sn}] We will show that each sum converges separately with a rate of at most $O(1/n)$. For the first one whose terms involve $\psi_{\la n, j+1}$, recall its definition at \eqref{def:psiZ}, the definition of $\phi$ at \eqref{def:phi}. Using the expansion given at \eqref{eq:gamma_expansion} at Page \pageref{eq:gamma_expansion}, we can then write
\begin{equation*}
    \begin{split}
       \frac{1}{(\la n+j+\al)\psi_{\la n,j+1}} & = \frac{\la n +j}{(\la n + j +\al)^2}\frac{\Gamma(\la n +2 + \al)\Gamma(\la n + j)}{(\la n + \al)\Gamma(\la n + 2)\Gamma(\la n + j+ \al)} \\
       & = \frac{\la^{\al}n^{\al}(\la n + j)}{(\la n + j + \al)^2(\la n + \al)(\la n +j)^\al}\left[ 1 + O\left(\frac{1}{\la n  }\right)\right]
    \end{split}
\end{equation*}
Now, observe that we also have that
\begin{equation*}
    \begin{split}
     \sum_{j=1}^n\frac{(\la n +j)^{1-\al}}{(\la n +j + \al)^2} & = \sum_{j=1}^n\frac{(\la + j/n)^{1-\al}}{n^{1+\al}(\la +j/n + \al/n)^2} \\
     & = \frac{1}{n^\al}\sum_{j=1}^n\frac{(\la + j/n)^{1-\al}}{n(\la +j/n + \al/n)^2} \\
     & =: \frac{1}{n^\al}R_n.
    \end{split}
\end{equation*}
On the other hand, notice that the term $R_n$ is a Riemann sum (right endpoint approximation) of the integral
$$
\int_0^1\frac{(\la +x)^{1-\al}}{(\la + x +\al/n)^2}{\rm d}x
$$
Moreover, 
\begin{equation*}
    \left| \frac{(\la +x)^{1-\al}}{(\la + x +\al/n)^2} - \frac{1}{(\la + x)^{1+\al}}\right| \le \frac{C}{n},
\end{equation*}
for a constant $C$ depending on $\la$ and $\al$ only. This and Lemma~\ref{l:integral_approx} imply that $R_n$ converges to $\int_0^1(\la + x)^{-1-\al}{\rm d}x$ at a rate of $O(1/n)$.  Putting al the pieces together, we have that
\begin{equation}\label{eq:firstsumconv}
    \begin{split}
        \al^{2}n\sum_{j=1}^n\frac{1}{(\la n+j+\al)\psi_{\la n,j+1}} & = \frac{\al^2\la^\al n^{1+\al}}{\la n +\al}\sum_{j=1}^n\frac{(\la n +j)^{1-\al}}{(\la n + j +\al)^2}(1+O(n^{-1})) \\
        & = \frac{\al^2\la^\al n^{1+\al}}{\la n +\al}\frac{R_n}{n^{\al}}(1+O(n^{-1}))\\
        & \to \frac{\al^2}{\la^{1-\al}}\int_{0}^1\frac{{\rm d}x}{(\la + x)^{1+\al}}\\
        & = \frac{\al^2}{\la^{1-\al}}\left(\frac{1}{\al\la^{\al}} - \frac{1}{\al (1+\la)^\al} \right),
    \end{split}
\end{equation}
which takes care of the sum. As for the second part, we can argue similarly to conclude, first that
$$
\al^2n\sum_{j=1}^n\frac{1}{(\la n + j + \al)^2} = \al^2 \widetilde{R}_n,
$$
where $\widetilde{R}_n$ is a Riemann sum is for the integral
$$
\int_0^1\frac{{\rm d}x}{(\la +x+\al/n)^2}.
$$
This allows us to conclude that
\begin{equation}\label{eq:}
\al^2n\sum_{j=1}^n\frac{1}{(\la n + j + \al)^2} \to \al^2\int_{0}^1\frac{{\rm d}x}{(\la +x)^2}= \al^2\left( \frac{1}{\la} - \frac{1}{1+\la} \right) = \frac{\al^2}{\la(1+\la)},
\end{equation}
which is enough to show the lemma, since
$$
\s_{\la,\al}^2 = \frac{\al^2}{\la^{1-\al}}\left(\frac{1}{\al\la^{\al}} - \frac{1}{\al (1+\la)^\al} \right) - \frac{\al^2}{\la(1+\la)}.
$$
\end{proof}

\appendix
\section{Classical Theorems}
\begin{theorem}[Theorem 3.9 in \cite{hallHeyde80}]\label{t:hallheyde} Let $\{S_n = \sum_{i=1}^nX_i, \F_i, \le 1\le i\le n$\} be a zero-mean martingale. Set 
$$
V_n^2 = \sum_{i=1}^n E[X_i^2\;| \F_{i-1}],
$$
and suppose that $0<\delta\le 1$. Define
$$
L_n = \sum_{i=1}^nE|X_i|^{2+2\delta} + E|V_n^2-1|^{1+\delta}.
$$
There exists a constant $A$ depending only on $\delta$, such that 
\begin{equation*}
    |P(S_n\le x) - \Phi(x)| \le AL_n^{1/(3+2\delta)}\left[1+|x|^{4(1+\delta)^2/(3+2\delta)}\right]^{-1}
\end{equation*}
    
\end{theorem}

\begin{theorem}[Theorem 3, Chapter V in \cite{brown2012sums}]\label{t:petrov} Let $X_1,X_2,\dots,X_n$ be independent random variables such that $EX_j = 0$, $E|X_j|^3 < \infty $, for $j\in \{1,2,\dots,n\}$. We write
$$
\s_j^2 = EX_j^2, \quad B_n = \sum_{j=1}^n\s_j^2, \quad F_n(x) = P\left(B_{n}^{-1/2}\sum_{j=1}^nX_j<x\right),
$$
and
$$
L_n = B^{-3/2}_n\sum_{j=1}^nE|X_j|^3.
$$
Then, there exists $A>0$ such that
$$
\sup_{x \in \mathbb{R}}|F_n(x) - \Phi(x)| \le AL_n
$$
\end{theorem}

\bibliographystyle{plain}
\bibliography{ref}

\end{document}